\documentclass{amsart}
\makeatletter
\@namedef{subjclassname@2020}{%
  \textup{2020} Mathematics Subject Classification}
\makeatother
\usepackage{amssymb}
\usepackage[utf8]{inputenc} 
\usepackage[T1]{fontenc}    
\usepackage{url}            
\usepackage{booktabs}       
\usepackage{amsfonts}       
\usepackage{nicefrac}       
\usepackage{microtype}      
\usepackage{color}
\usepackage{lipsum}
\usepackage{calrsfs}
\usepackage{enumitem}
\usepackage{tikz-cd}
\usepackage{tikz}
\usepackage{comment}

\usetikzlibrary{arrows}

\DeclareMathAlphabet{\pazocal}{OMS}{zplm}{m}{n}

\usepackage{eqnarray,amsmath}
\usepackage{marginnote}

\usepackage{hyperref}
\hypersetup{
    colorlinks=true,
    linkcolor=blue,
    filecolor=magenta,      
    urlcolor=cyan,
    pdftitle={Overleaf Example},
    pdfpagemode=FullScreen,
    }

\def\N{\mathbb N}

\def\SPAN{\mathop{\hbox{\rm span}}}

\def\h{\mathop{\hbox{\rm h}_{\hbox{\tiny\rm alg}}}}

\def\remove#1{}
\def\F{\mathcal{F}}
\def\G{\mathcal{G}}
\def\len{\mathop{\ell}}
\def\reg{\mathop{\hbox{Reg}}}

\newtheorem{lemma}{Lemma}[section]

\newtheorem{theorem}[lemma]{Theorem}
\newtheorem{proposition}[lemma]{Proposition}
\newtheorem{remark}[lemma]{Remark}

\newtheorem*{theorem*}{Theorem}

\theoremstyle{definition}
\newtheorem{definition}[lemma]{Definition}

\title[The algebraic entropies of $L_K(E)$ and  $KE$ agree]{The algebraic entropies of the Leavitt path algebra and the graph algebra agree}

\author[Bock]{Wolfgang Bock}
\address{W. Bock: Linnaeus University, Department of Mathematics, V\"axj\"o, Sweden. }
\email{wolfgang.bock@lnu.se}

\author[Gil]{Crist\'obal Gil Canto}
\address{C. Gil Canto: Departamento de Matem\'atica Aplicada, E.T.S. Ingenier\'\i a Inform\'atica, Universidad de M\'alaga, 
M\'alaga,   Spain.}
\email{cgilc@uma.es}

\author[Mart\'in]{Dolores Mart\'in Barquero} 
\address{D. Mart\'{\i}n Barquero: Departamento de Matem\'atica Aplicada, Escuela de Ingenier\'{\i}as Industriales, Universidad de M\'alaga, 
M\'alaga, Spain.}
\email{dmartin@uma.es}

\author[Mart\'in]{C\'andido Mart\'in Gonz\'alez} 
\address{C. Mart\'{\i}n Gonz\'alez:  Departamento de \'Algebra Geometr\'{\i}a y Topolog\'{\i}a, Fa\-cultad de Ciencias, Universidad de M\'alaga, 
M\'alaga, Spain.}
\email{candido\_m@uma.es}

\author[Ruiz]{Iv\'an Ruiz Campos}
\address{I. Ruiz Campos:  Departamento de \'Algebra Geometr\'{\i}a y Topolog\'{\i}a, Fa\-cultad de Ciencias, Universidad de M\'alaga, 
M\'alaga, Spain.}
\email{ivaruicam@uma.es}

\author[Sebandal]{Alfilgen Sebandal}
\address{A. Sebandal:  Research Center for Theoretical Physics in Jagna Bohol, Philippines}
\email{a.sebandal@rctpjagna.com}


\begin{document}
\maketitle
\begin{abstract}
 In this note we prove that the algebras $L_K(E)$ and $KE$ have the same entropy. Entropy is always referred to the standard filtrations in the corresponding kind of algebra. The main argument leans on (1) the holomorphic functional calculus; 
 (2) the relation of entropy with suitable norm of the adjacency matrix;
 and (3) the Cohn path algebras which yield suitable bounds for the algebraic entropies.  
\end{abstract}

\section{Introduction}
We have applied the notion of algebraic entropy on  path, Cohn and Leavitt path algebras in \cite{entropy1}. There, we saw how the entropy is strongly dependent on the filtrations used in each case. The paper \cite{entropy1} also contains a trichotomy which may shed light into classifications tasks on these algebras. At the end of the paper we detected empirically a fact: the entropies of $KE$ and $L_K(E)$ seem to coincide. We illustrated this phenomenom with a number of computer-aided examples. In this note we prove that the algebraic entropies of both algebras agree. 

The paper is organized as follows. After the preliminaries section in which we recall the standard filtrations defined in \cite{entropy1} on the path algebra $KE$ as well as in $C_K(E)$ and $L_K(E)$,  we prove, in Section \ref{eltecnico},  the fundamental bounds relating the entropies of $KE$, $L_K(E)$ and $C_K(E)$. Then Lemma \ref{nueces} gives an expression of entropy in terms of norms of powers of adjacency matrices. Next, we recall the fundamental formula in holomorphic functional calculus, applied to matrices and we use it to relate  $\h(C_K(E))$ and $\h(KE)$. Finally, Theorem \ref{almendras} gives the desired result. 

\section{Preliminaries}

\noindent A \emph{directed graph, digraph or quiver} is a $4$-tuple $E=(E^0, E^1, s, r)$ 
consisting of two disjoint sets $E^0$, $E^1$ and two maps
 $r, s: E^1 \to E^0$. The elements of $E^0$ are called \emph{vertices} and the elements of $E^1$ are called \emph{edges} of $E$. We say that $E$ is a \emph{finite graph} if $\vert E^0 \cup E^1\vert < \infty$. For $e\in E^1$, $s(e)$ and $r(e)$ is 
 the \emph{source} and the \emph{range} of $e$, respectively. A
vertex $v$ for which $s^{-1}(v)=\emptyset$  is called a \emph{sink}, while a vertex $v$ for which $r^{-1}(v)=\emptyset$ is called a \emph{source}. 
 We will denote the set of sinks of $E$ by $\text{Sink}(E)$ and the set of sources by $\text{Source}(E)$. We say that a vertex $v \in E^0$ is a \emph{infinite emitter} if $\vert s^{-1}(v)\vert =\infty$. The \emph{set of regular
vertices} (those which are neither sinks nor infinite emitters) is denoted by $\text{Reg}(E
)$. A graph $E$ is \emph{row-finite} if  $s^{-1}(v)$ is a finite set for every $v \in E^0$. Throughout this paper we only consider finite graphs. We say that a \emph{path} $\mu$ has \emph{length} $m \in \N$, denoted by $l(\mu)=m$, if it is a finite chain of edges $\mu=e_1\ldots e_m$ such that $r(e_i)=s(e_{i+1})$ for $i=1,\ldots,m-1$.  We define $\text{Path}(E)$ as the set of all paths in $E$. We denote by $s(\mu):=s(e_1)$ the source of $\mu$ and $r(\mu):=r(e_m)$ the range of $\mu$.  We write $\mu^0$ the set of vertices of $\mu$. The vertices are the trivial paths. 


\indent For a directed graph $E$ and a field $K$, the  \textit{path algebra} of $E$, denoted by $KE$, is the $K$-algebra generated by the sets $\{v:v\in E^0\}$ and $\{e \colon e \in E^1  \}$ with coefficients in $K$, subject to the relations:
\begin{enumerate}
\item[\textnormal{(V)}]
$v_iv_j=\delta_{i,j}v_i$ for every $v_i, v_j\in E^0$;
\item[\textnormal{(E)}] $s(e)e=e = e r(e)$ for all non-sinks $e \in E^1$.
\end{enumerate}

 The \emph{extended graph} of $E$ is defined as the new graph $\widehat{E}=(E^0, E^1\cup (E^1)^*, r',s'  )$, where $(E^1)^*= \{ e^*:e\in E^1 \}$ and the maps $r'$ and $s'$ are defined as
$r'|_{E^1}=r$, $s'|_{E^1}=s$, $r'(e^*)=s(e)$, and $s'(e^*)=r(e)$ for all $e\in E^1$. In other words, each $e^*\in (E^1)^*$ has orientation the reverse of that of its counterpart $e\in E^1$. The elements of $(E^1)^*$ are called $ghost$ $edges$.

 \indent For a  directed graph $E$ and a field $K$, the  \textit{Leavitt path algebra} of $E$, denoted by $L_K(E)$, is the path algebra over the extended graph $\hat{E}$ with additional relations
\begin{enumerate}
\item[\textnormal{(CK1)}] $e^*e'=\delta_{e,e'}r(e)$ for all $e, e'\in E^1$;

\item[\textnormal{(CK2)}]$\sum_{ \{e \in E^1 \colon s(e)=v    \}  } e e^*=v$ for every $v\in \text{Reg}(E)$.

\end{enumerate}

Let $E$ be an arbitrary directed graph and $K$ any field. The Cohn path algebra of $E$ with coefficients in $K$, denoted by $C_K(E)$, is the path algebra over the extended graph $\hat{E}$ with  additional relation given in (CK1).
We will construct certain quotient algebras of Cohn path algebras. For
$v \in \reg(E)$, consider the element $q_v$ of $C_K(E)$
$$q_v= v-\sum_{e \in s^{-1}(v)}
ee^*.$$
 Let $X$ be any subset of
$\reg(E)$. We denote by $I^X$ the $K$-algebra ideal of $C_K(E)$ generated by the idempotents $\{q_v \colon  v \in X\}$. The Cohn path algebra of $E$ relative to $X$, denoted by $C_K^X(E)$, is defined
to be the quotient $K$-algebra
$C_K(E)/I^X$.
This definition of the relative Cohn path algebra relates the Cohn path algebra with the Leavitt
path algebra, that to say,
$C_K(E)= C_K^{\emptyset}(E)$ and $L_K(E)=  C_K^{\tiny \reg(E)}(E)$.
Other basic definitions and results on graphs, Cohn and Leavitt path algebras can be seen in the book \cite{AAS}.

Recall that, as in \cite{entropy1},
a $K$-algebra $A$ is said to be \emph{filtered} if it is endowed with a collection of subspaces $\mathcal{F}=\{V_n\}_{n=0}^\infty$ such that

\begin{enumerate}
\item $0\subset V_0\subset V_1\subset\cdots\subset V_n\subset V_{n+1}\subset\cdots A$,  
\item $A=\cup_{n\ge 0}V_n$,
\item $V_nV_m\subset V_{n+m}$.
\end{enumerate}

Hence, we define as in \cite{entropy1} the \emph{algebraic entropy of a filtered algebra} $(A,\F)$, $$\h(A,\F):=
\begin{cases}
0 &\text{ if } $A$ \text{ is finite dimensional,} \\
\displaystyle \limsup_{n\to\infty}\frac{
\log\dim(V_n/V_{n-1})}{n} & \text{ otherwise.}
\end{cases}$$

If there is no doubt about the filtration $\F$ that we are considering in $A$, then we can shorten the notation 
$\h(A,\F)$ to $\h(A)$.

\begin{definition}\cite{entropy1}\label{churros} \rm For $KE$ we define the filtration $\{V_i\}_{i\in\N}$ where $V_0$ is the linear span of the set of vertices of the graph $E$, while
$V_1$ is the sum of $V_0$ with the linear span of the set of edges, and $V_{k+1}$ linear span of the set of paths of length less or equal to $k+1$. This will be termed, the {\em standard filtration on $KE$}. 
\end{definition}
\begin{definition}\cite{entropy1}\rm\label{teje}
For $L_K(E)$ we define its {\it standard filtration} $\{W_i\}_{i \in \mathbb{N}}$  as the one for which  $W_0$ is  the linear span of the set of vertices of $E$ and $W_1$ the sum of $W_0$ plus the linear span of the set
$E^1\cup (E^1)^*$. For $W_{k}$ we take the linear span of the set of elements:  $\lambda\mu^*$ with $l(\lambda)+l(\mu)\le k$. 
\end{definition} 

\begin{definition}\rm
    Let $E$ be a finite directed graph, we define the \emph{standard filtration} of $C_K^X(E)$ as the families of vector spaces $\F:=\{V_n\}_{n\geq 0}$ where
\begin{equation*}
\begin{split}
V_0 & := \SPAN(E^0), \\
V_1 & := \SPAN(E^1 \cup (E^1)^*), \\
     V_n & := \SPAN \{\lambda \mu^* \in C_K^X(E) \colon \len(\lambda) + \len(\mu) \leq n\}.
\end{split}
    \end{equation*}
\end{definition}

From now on, any path algebra $KE$ will be understood to be endowed with its standard filtration by default (and the same applies to $L_K(E)$ and $C_K^X(E)$).

\section{Agreement of the algebraic entropies of $L_K(E)$ and $KE$}\label{eltecnico}
In this section we prove the main results. To start with, we relate the entropies of $KE$ and $C_K^X(E)$ by proving that one is bounded above by the other.

\begin{proposition}
    Let $E$ be a finite directed graph. Then
    \begin{equation*}
        \h(KE) \leq \h(C_K^X(E)).
    \end{equation*}
\end{proposition}
\begin{proof}
   Let us consider the canonical monomorphism $i \colon KE \hookrightarrow C_K^X(E)$ and the standard filtrations $\F = \{V_n\}_{n\geq 0}$ and $\G = \{W_n\}_{n\geq 0}$ respectively. We observe that $i(V_n) = i(KE) \cap W_n$ and applying \cite[Lemma 4.2]{entropy1} the conclusion follows. 
\end{proof}

\begin{remark}\label{aceitunas} \rm If $E$ is a finite directed graph, and $X \subseteq Y \subseteq \reg(E)$, then $C_K^Y(E) \simeq C_K^X(E)/I$ for a suitable ideal generated by the new relations we need to add due to $(CK2)$. In particular, there is a projection $p \colon C_K^X(E) \to C_K^Y(E)$ with $p(\F) = \G$ where $\F$ and $\G$ are the standard filtrations of $C_K^X(E)$ and $C_K^Y(E)$ respectively. As a consequence of \cite[Lemma 4.1]{entropy1}, we have that
\begin{equation}
    \h(C_K^Y(E)) \leq \h(C_K^X(E)).
\end{equation}
In particular, for $X=\emptyset$ and $Y= \reg(E)$ we have $$\h(L_K(E)) \leq \h(C_K(E)).$$
\end{remark}

We can describe the dimension of $V_{k+1}/V_{k}$ for a Cohn path algebra in terms of the following matrix norm. If $A = (a_{i,j})_{i,j = 1}^n$ is a matrix, then we define 
\begin{equation*}
    \|A \| = \sum_{i,j = 1}^n |a_{i,j}|.
\end{equation*}

\begin{lemma}\label{nueces} If $E$ is a finite directed graph with $A_E$ its corresponding adjacency matrix and transpose adjacency matrix denoted by $A_E^*$, then $$\h(C_K(E))=\lim_{k \to \infty}\frac{1}{k}\sum_{s=0}^k \|A_E^s(A_E^*)^{k-s}\|.$$
\end{lemma}

\begin{proof} By the first part of the formula given in \cite[Theorem 40]{Adjacency} (also in \cite[Proposition 6.1]{entropy1}) we have the number of linear independent elements $\lambda \mu^*$ for $C_K(E)$ such that $\len(\lambda) + \len(\mu)=k$ where $k \in \mathbb{N}$, $\lambda,\mu \in {\rm Path}(E)$, that is, $$ {\rm dim}(V_{k+1}/V_{k})= \sum_{s=0,j=1}^{k,n} \left (\sum_{i=1}^n (A_E^s)_{i,j} \right ) \left (\sum_{i=1}^n(A_E^{k-s})_{i,j} \right ).$$ 
But
\begin{equation*}
\begin{split}
       \sum_{s=0,j=1}^{k,n} \left (\sum_{i=1}^n (A_E^s)_{i,j} \right ) \left (\sum_{i=1}^n(A_E^{k-s})_{i,j} \right ) & =  \sum_{s=0,i,j,l=1}^{k,n} (A_E^s)_{i,j} (A_E^{k-s})_{l,j} = \\
 \sum_{s=0,i,l=1}^{k,n}\sum_{j=1}^n(A_E^s)_{i,j} ((A_E^*)^{k-s})_{j,l} &= \sum_{s=0}^k\sum_{i,l=1}^n (A_E^s(A_E^*)^{k-s})_{i,l} = \\
    \sum_{s=0}^k \|A_E^s(A_E^*)^{k-s}\|. & 
\end{split}
\end{equation*}
\end{proof}

Before proving that the entropies of a Leavitt path algebra and of a path algebra agree, we recall the following well-known result about holomorphic functional calculus (see \cite[ Definition 10.26, formula (2) and  Theorem 10.27]{Rudin}).

\begin{theorem}\cite[Theorem 10.27]{Rudin}\label{resaca}
If $A \in \mathcal{M}_n(\mathbb{C})$, then for every $k \geq 1$
$$A^k = \frac{r^{k+1}}{2 \pi} \int_{0}^{2 \pi} e^{it(k+1)} \frac{1}{re^{it}-A} \,dt,$$
where $r$ is larger than the spectral radius $\rho(A)$.
\end{theorem}
\begin{theorem}\label{almendras} If $E$ is a finite directed graph, then $\h(C_K(E)) \leq \h(KE)$. Consequently,
$$\h(KE) = \h(L_K(E)) = \h(C_K(E)).$$
\end{theorem}
\begin{proof} Write $A:= A_E$ the adjacency matrix of $E$, $A^*$ the corresponding transpose matrix of $A$ and suppose $r > \rho(A)$. First we compute $\displaystyle \left \| \frac{1}{re^{it}-A} \right \|$. On the one hand,
$\displaystyle \frac{1}{re^{it}-A}=\frac{1}{r}e^{-it}\left (I+\frac{A}{re^{it}}+\frac{A^2}{(re^{it})^2}+ \cdots \right)$ if $\|A\| <r$. So
$\displaystyle \left \|\frac{1}{re^{it}-A} \right \|=\frac{1}{r} \left \| I+\frac{A}{re^{it}}+\frac{A^2}{(re^{it})^2}+ \cdots \right \| \leq \frac{1}{r} \left (\|I\| + \frac{\|A\|}{r} + \frac{\|A\|^2}{r^2}+\cdots \right )=\frac{1}{r}\left(n+\frac{\frac{\|A\|}{r}}{1-\frac{\|A\|}{r}} \right)=\frac{1}{r} \left (n+\frac{\|A\|}{r-\|A\|} \right )$.
On the other hand, we bound $\displaystyle \left \|\int_{0}^{2 \pi} e^{it(k+1)} \frac{1}{re^{it}-A} \,dt \right \|$. In fact, we have $\displaystyle \left \|\int_{0}^{2 \pi} e^{it(k+1)} \frac{1}{re^{it}-A} \,dt \right \| \leq \int_{0}^{2 \pi} \left \|\frac{1}{re^{it}-A} \right \| \, dt \leq \frac{1}{r}\left (n+\frac{\|A\|}{r-\|A\|} \right ) 2\pi$ taking into account the previous computation. Hence by Theorem \ref{resaca}, for every $k \geq 1$, we obtain
$$\|A^k\| \leq \frac{r^{k+1}}{2 \pi}\frac{1}{r}\left(n+\frac{\|A\|}{r-\|A\|} \right) 2\pi = r^k\left(n+\frac{\|A\|}{r-\|A\|}\right).$$
At this point recall that by Lemma \ref{nueces}, $\h(C_K(E))=\lim_{k \to \infty}\frac{1}{k}\sum_{s=0}^k \|A^s(A^*)^{k-s}\|$. Consider $\|A^s(A^*)^{k-s}\|$. Since ${\rm Spec}(A)={\rm Spec}(A^*)$ we have that $\displaystyle\|A^s(A^*)^{k-s}\| \leq \|A^s\|\|(A^*)^{k-s}\| \leq r^s \left(n+\frac{\|A\|}{r-\|A\|}\right) r^{k-s}\left(n+\frac{\|A^*\|}{r-\|A^*\|}\right)=r^k\left(n+\frac{\|A\|}{r-\|A\|}\right)^2$. Therefore, $\displaystyle\sum_{s=0}^k \|A^s(A^*)^{k-s}\| \leq r^k\left(n+\frac{\|A\|}{r-\|A\|}\right)^2 (k+1)$. Finally, $$\log \left(\sum_{s=0}^k \|A^s(A^*)^{k-s}\|\right) \leq \log(k+1) + k\log(r) + 2\log\left(n+\frac{\|A\|}{r-\|A\|}\right) \iff$$
$$\frac{1}{k}\log \left(\sum_{s=0}^k \|A^s(A^*)^{k-s}\|\right) \leq \frac{\log(k+1)}{k} + \log(r) + \frac{2}{k}\log\left(n+\frac{\|A\|}{r-\|A\|}\right),$$
which means,
$\h(C_K(E)) \leq \log(r)$ (whenever $r > \rho(A)$).
In particular, $$\h(C_K(E)) \leq \log(\rho(A)) = \h(KE)$$ by \cite[Theorem 5.4]{entropy1}. 

Finally, taking also into account \cite[Lemma 4.2]{entropy1} and Remark \ref{aceitunas} we have that for every finite graph $E$, $$\h(KE) \leq \h(L_K(E)) \leq \h(C_K(E)) \leq \h(KE).$$ 
\end{proof}

\section*{Acknowledgements}{The second, third, fourth  and fifth authors are supported by the Spanish Ministerio de Ciencia e Innovaci\'on   through project PID2023-152673NB-I00 
and by the Junta de Andaluc\'{i}a  through project  FQM-336,  all of them with FEDER funds. The fifth author is supported by a Junta de Andalucía PID fellowship no. PREDOC\_00029.}

\section*{Conflict of interest statement}
All authors declare that they have no conflicts of interest. 

\section*{Data Availability Statement}
There is no data used in throughout this article.

\bibliographystyle{plain}
\bibliography{ref}
\end{document}